\documentclass[12pt]{article}

\usepackage{amsmath,amsthm,amsfonts,latexsym,amsopn,verbatim,amscd,amssymb,}
\usepackage{amssymb}
\usepackage{amsfonts}

\usepackage{graphics,color}
\usepackage{graphicx}

\theoremstyle{plain}
\newtheorem{Thm}{Theorem}[section]
\newtheorem{Lem}[Thm]{Lemma}
\newtheorem{Prop}[Thm]{Proposition}
\newtheorem{Cor}[Thm]{Corollary}
\theoremstyle{definition}

\newtheorem{Def}[Thm]{Definition}
\newtheorem{example}[Thm]{Example}

\numberwithin{equation}{section}

\newcommand{\bnum}{\begin{enumerate}}
\newcommand{\enum}{\end{enumerate}}
\begin{document}
\begin{center}
\textbf{Weakly $r$-clean rings and weakly $\star$-clean rings}\\
\end{center}
\begin{center}
Ajay Sharma \\
\small{\it Department of Mathematical Sciences, Tezpur University,
 \\ Napaam, Tezpur-784028, Assam, India.\\
Email: ajay123@tezu.ernet.in}
\end{center}

\begin{center}
Dhiren Kumar Basnet\\
\small{\it Department of Mathematical Sciences, Tezpur University,
 \\ Napaam, Tezpur-784028, Assam, India.\\
Email: dbasnet@tezu.ernet.in}

\end{center}
\noindent \textit{\small{\textbf{Abstract:}  }} Motivated by the concept of weakly clean rings, we introduce the concept of weakly $r$-clean rings. We define an element $x$ of a ring $R$ as weakly $r$-clean if it can be expressed as $x=r+e$ or $x=r-e$ where $e$ is an idempotent and $r$ is a regular element of $R$. If all the elements of $R$ are weakly $r$-clean then $R$ is called a weakly $r$-clean ring. We discuss some of its properties in this article. Also we generalise this concept of weakly $r$-clean ring to weakly $g(x)$-$r$-clean ring, where $g(x) \in C(R)[x]$ and $C(R)$ is the centre of the ring $R$. Finally we introduce the concept of weakly $\star$-clean and $\star$-$r$-clean ring and discuss some of their properties. 
\bigskip

\smallskip

\noindent \small{\textbf{\textit{$2010$ Mathematics Subject Classification:}}  16N40, 16U99.} \\
\smallskip

\bigskip

\section{INTRODUCTION}
$\mbox{\hspace{.5cm}}$ Here rings $R$ are associative with unity unless otherwise indicated. The Jacobson radical, set of units, set of idempotents, set of regular elements, set of nilpotent elements and centre of a ring $R$ are denoted by $J(R)$, $U(R)$, $Idem(R)$, $Reg(R)$, $Nil(R)$ and $C(R)$ respectively. Nicholson\cite{8} called an element $x$ of  $R$, a clean element, if $x=e+u$ for some $e\in Idem(R)$, $u\in U(R)$ and called the ring $R$ as clean ring if all its elements are clean. Weakening the condition of clean element, M.S. Ahn and D.D. Anderson\cite{1} defined  an element $x$ as weakly clean if $x$ can be expressed as $x=u+e$ or $x=u-e$, where $u\in U(R)$, $e\in Idem(R)$. Generalising the idea of clean element, N. Ashrafi and E. Nasibi \cite{2, 3} introduced the concept of $r$-clean element as follows: an element $x$ of a ring $R$ is said to be $r$-clean if it can be written as a sum of an idempotent and a regular element. Further if all the elements of $R$ are $r$-clean then the ring $R$ is called $r$-clean ring.  Motivated by these ideas we define an element of a ring $R$ as weakly $r$-clean if $x$ can be expressed as $x=r+e$ or $x=r-e$, where $r\in Reg(R)$, $e\in Idem(R)$ and call the ring $R$ to be weakly $r$-clean if all its elements are so. We show that in case of an abelian ring the concept of weakly $r$-clean rings coincides with that of weakly clean rings and hence with that of weakly exchange rings. We also show that the center of weakly $r$-clean rings are not so. However if $R$ has no non-trivial idempotents then  the center of a weakly $r$-clean ring is weakly $r$-clean. Further we discuss some interesting properties of weakly $g(x)$-$r$- clean ring, where $g(x) \in C(R)[x]$. Finally we define the concept of weakly $\star$-clean ring and $\star$-$r$-clean ring. A ring $R$ is a $\star$-ring (or ring with involution) if there exists an operation $\star \,\,:\,\,R\rightarrow R$ such that for all $x,y\in R$, $(x+y)^{\star}=x^{\star}+y^{\star}$, $(xy)^{\star}=y^{\star}x^{\star}$ and $(x^{\star})^{\star}=x$. An element $p$ of a $\star$-ring is a projection if $p^2=p=p^{\star}$. Obviously, $0$ and $1$ are projections of any $\star$-ring. Henceforth $P(R)$ will denote the set of all projections in a $\star$-ring. We define an element  $x$ in a $\star$-ring $R$ is weakly $\star$-clean element if $x=u+p$ or $x=u-p$, where $u\in U(R)$ and $p\in P(R)$. If all the elements of $R$ are weakly $\star$-clean then the $\star$-ring $R$ is called weakly $\star$-clean. We show that for an abelian $\star$-clean ring $R$, if every idempotent of the form $e=ry$ or $e=yr$ is projection, where $r\in Reg(R)$, then $R$ is $\star$-$r$-clean if and only if $R$ is $\star$-clean.

\section{Weakly $r$-clean ring}
\begin{Def}
  An element $x$ in a ring $R$ is called weakly $r$-clean if it can be written as  $x=r+e$ or $x=r-e$, for some $r\in Reg(R)$ and $e\in Idem(R)$. If all the elements of $R$ are weakly $r$-clean, then the ring $R$ is called a weakly $r$-clean ring.
\end{Def}
For example all clean rings, weakly clean rings and $r$-clean rings are weakly $r$-clean ring. The following theorem is obvious.
\begin{Thm}
  Homomorphic image of weakly $r$-clean ring is weakly $r$-clean.
\end{Thm}
However the converse is not true as $\mathbb{Z}_6\cong \mathbb{Z}/\langle 6 \rangle$ is a weakly $r$-clean ring, but $\mathbb{Z}$ is not weakly $r$-clean ring.
\begin{Thm}
  Let $\{R_\alpha\}$ be a collection of rings. Then the direct product $R=\underset{\alpha}{\prod}R_{\alpha}$ is weakly $r$-clean \textit{if and only if} each $R_{\alpha}$ is weakly $r$-clean ring and at most one $R_{\alpha}$ is not $r$-clean.
\end{Thm}
\begin{proof}
  Let $R$ be weakly $r$-clean ring. Then being homomorphic image of $R$ each $R_{\alpha}$ is weakly $r$-clean. Suppose $R_{\alpha_1}$ and $R_{\alpha_2}$ are not $r$-clean, where ${\alpha}_1\neq {\alpha}_2$. Since $R_{\alpha_1}$ is not $r$-clean so not all elements $x\in R_{\alpha_1}$ of the form $x=r-e$ where $r\in Reg(R_{\alpha_1})$ and $e\in Idem(R_{\alpha_1})$. As $R_{\alpha_1}$ is weakly $r$-clean so there exists $x_{\alpha_1}\in R_{\alpha_1}$ with $x_{\alpha_1}=r_{\alpha_1}+e_{\alpha_1}$, where $r_{\alpha_1}\in Reg(R_{\alpha_1})$ and $e_{\alpha_1}\in Idem(R_{\alpha_1})$ but $x_{\alpha_1}\neq r-e$ for any $r\in Reg(R_{\alpha_1})$ and $e\in Idem(R_{\alpha_1})$. Similarly there exists $x_{\alpha_2}\in R_{\alpha_2}$ with $x_{\alpha_2}=r_{\alpha_2}-e_{\alpha_2}$, where $r_{\alpha_2}\in Reg(R_{\alpha_2})$ and $e_{\alpha_2}\in Idem(R_{\alpha_2})$, but $x_{\alpha_2}\neq r+e$ for any $r\in Reg(R_{\alpha_2})$ and $e\in Idem(R_{\alpha_2})$. Define $x=(x_\alpha)\in R$ by \begin{align*}
    x_\alpha &=x_{\alpha}\,\, \,\,\,\,\,\,\,if\,\, \alpha \in \{\alpha_1,\alpha_2\} \\
             &=0 \,\,\,\,\,\,\,\,\,\,\,\,\,if \,\,\alpha \notin \{\alpha_1,\alpha_2\}
  \end{align*}
  Then clearly $x\neq r\pm e$ for any $r\in Reg(R)$ and $e\in Idem(R)$. Hence at most one $R_\alpha$ is not $r$-clean.\\
  $(\Leftarrow)$ If each $R_\alpha$ is $r$-clean then $R=\prod R_\alpha$ is $r$-clean and hence weakly $r$-clean. Assume $R_{\alpha_0}$ is weakly $r$-clean but not $r$-clean and that all other $R_{\alpha}$'s are $r$-clean. If $x=(x_\alpha) \in R$ then in $R_{\alpha_0}$ we can write $x_{\alpha_0}=r_{\alpha_0}+e_{\alpha_0}$ or $x_{\alpha_0}=r_{\alpha_0}-e_{\alpha_0}$, where $r_{\alpha_0}\in Reg(R_{\alpha_0})$ and $e_{\alpha_0}\in Idem(R_{\alpha_0})$. If $x_{\alpha_0}=r_{\alpha_0}+e_{\alpha_0}$ then for $\alpha \neq \alpha_0$ let, $x_{\alpha}=r_{\alpha}+e_{\alpha}$ and if $x_{\alpha_0}=r_{\alpha_0}-e_{\alpha_0}$ then for $\alpha \neq \alpha_0$ let, $x_{\alpha}=r_{\alpha}-e_{\alpha}$ then $r=(r_\alpha)\in Reg(R)$ and $e=(e_\alpha)\in Idem(R)$ such that $x=r+e$ or $x=r-e$ and consequently $R$ is weakly $r$-clean ring.
\end{proof}

\begin{Lem}
  Let $R$ be a ring with no zero divisor. Then $R$ is weakly clean if and only if $R$ is weakly $r$-clean.
\end{Lem}
\begin{proof}
  Let $r(\neq 0)\in Reg(R)$, then $r=ryr$ for some $y\in R$, \textit{i.e.}, $r(1-yr)=0$, which implies that $r$ is a unit. Now the result follows immediately.
\end{proof}
\begin{Lem} \label{10}
  Let $R$ be a ring and $e^2=e\in R$, $a=u-f \in eRe$, where $u\in U(eRe)$, $f\in Idem(eRe)$. Then there exist elements $v\in U(R)$ and $\overline{e}\in Idem(R)$ such that $a=v-\overline{e}$.
\end{Lem}
\begin{proof}
  Since $u\in U(eRe)$, so there exists $w\in U(eRe)$ such that $uw=e=wu$. Assume $v=u+(1-e)$, then $v(w+(1-e))=1=(w+(1-e))v$, because $v(1-e)=0=(1-e)w$.
   Clearly $ a-v=a-u-(1-e)=-(f+(1-e))$. It is easy to see that $\overline{e}=f+(1-e)$ is idempotent and hence the result follows.
\end{proof}
For the sake of completeness of the article we prove a result on weakly clean element.
\begin{Prop}
  Let $R$ be an abelian ring and $a\in R$ is weakly clean element then $ae$ is weakly clean for any idempotent $e\in R$.
\end{Prop}
\begin{proof}
  Let $a\in R$ be a weakly clean element. So $a=u+e'$ or $a=u-e'$, where $u\in U(R)$ and $e'\in Idem(R)$\\
  Case $1$: If $a=u+e'$, then  $ae$ is clean for any idempotent $e$ by Lemma 2.1 \cite{2}. \\
  Case $2$: If $a=u-e'$, then $ae=ue-e'e$. But $ue=ue^2=eue\in U(eRe)$ and $e'e\in Idem(eRe)$, as $R$ is abelian. So by Lemma {\ref{10}},  $ae=v-\overline{e}$, for some $v\in U(R)$ and $\overline{e}\in Idem(R)$. Hence $ae$ is weakly clean.
\end{proof}
\begin{Lem}
  Let $R$ be an abelian ring such that  $a$ and $-a$ both are clean. Then the followings hold:
  \begin{enumerate}
    \item $a+e$ is clean for any idempotent $e$ of $R$.
    \item $a-e$ is weakly clean for any idempotent $e$ of $R$.
  \end{enumerate}
\end{Lem}
\begin{proof}
  $(i)$ It follows from Lemma 2.1 \cite{2} .\\
  $(ii)$ Consider $a=u+e'$ and $-a=u'+e''$ where, $u,u'\in U(R)$ and $e,e'\in Idem(R)$. Now
  \begin{align*}
    a-e &=(ae+a(1-e))-e \\
        &=(a-1)e+a(1-e) \\
        &=-(1-a)e+a(1-e) \\
        &=-(1-e'-u)e+(-u'-e'')(1-e) \\
        &=-\{(1-e')e+e''(1-e)\}+\{ue-u'(1-e)\}
  \end{align*}
Since $\{(1-e')e+e''(1-e)\} \in Idem(R)$ and $\{ue-u'(1-e)\}\in U(R)$, so $a-e$ is weakly clean.
\end{proof}
 Here we define a weakly $r$-clean (respectively weakly clean) element $x\in R$ is of type $1$ form if $x=r+e$, where $r\in Reg(R)$ (respectively $r\in U(R)$) and $e\in Idem(R)$ and of type $2$ form if $x=r-e$, where $r\in Reg(R)$ (respectively $r\in U(R)$) and $e\in Idem(R)$.
 \begin{Thm}
   Let $R$ be an abelian ring. Then $R$ is weakly clean if and only if $R$ is weakly $r$-clean.
 \end{Thm}
 \begin{proof}
   $(\Rightarrow)$ Obviously weakly clean rings are weakly $r$-clean.\\
   $(\Leftarrow)$ Let $R$ be weakly $r$-clean ring and $x\in R$. So $x=r\pm e'$, where $r=ryr$ for some $y\in R$ and $e'\in Idem(R)$.\\
  First let $x=r+e'$. Assume $yr=e$, then clearly $e$ is an idempotent and therefore \\ $(ye+(1-e))(re+(1-e))=1$. Let $u=re+(1-e)$, then $u\in U(R)$ and $ue=r$, so $ue+f \in U(R)$, where $(1-e)=f$. Thus $ -r=-(ue+f)+f$ is clean. Hence by  lemma 2.7, $r+e'$ is clean.\\
   If $x=r-e'$, then clearly $r$ and $-r$ both are clean and so by lemma 2.7, $r-e'$ is weakly clean. Hence $R$ is weakly clean.
 \end{proof}

Consider the ring $\mathbb{Z}_{(6)}$  of all rational numbers $\frac{a}{b}$ of lowest form, where $b$ is relatively prime to $6$. Since $\frac{21}{11}+1=\frac{32}{11}$ and $\frac{21}{11}-1=\frac{10}{11}$ both are not units and 0, 1 are the only idempotents of $\mathbb{Z}_{(6)}$ so $\mathbb{Z}_{(6)}$ is not weakly clean and hence  not weakly $r$-clean. The following is an example of weakly $r$-clean but not $r$-clean ring.
\begin{example}
  $\mathbb{Z}_{(3)}\cap \mathbb{Z}_{(5)}$ is weakly $r$-clean but not $r$-clean.

   Since $\mathbb{Z}_{(3)}\cap \mathbb{Z}_{(5)}$ is weakly clean but not clean \cite{1} and if $R$ is an abelian ring then $R$ is  clean if and only if $R$ is  $r$-clean by   Theorem 2.2 \cite{2}. So by Theorem 2.8, $\mathbb{Z}_{(3)}\cap \mathbb{Z}_{(5)}$ is weakly $r$-clean but not $r$-clean.
\end{example}
 \begin{Def}
   A ring $R$ is called weakly exchange ring if for any $x\in R$, there exists an idempotent $e\in xR$ such that  $1-e \in (1-x)R$ or $1-e \in (1+x)R$.
 \end{Def}
 The relation between weakly $r$-clean ring and weakly exchange ring is given below.
 \begin{Thm}
   If $R$ is abelian ring then $R$ is weakly $r$-clean ring \textit{if and only if} $R$ is weakly exchange ring.
 \end{Thm}
 \begin{proof}
   $(\Rightarrow)$ It is equivalent to show $R$ is weakly clean ring \textit{if and only if} $R$ is weakly exchange ring.\\ Let $x\in R$, if $x=u+e$ where $u\in U(R)$ and $e\in Id(R)$ then clearly it satisfies the exchange property  \cite{8}. \\
   Suppose, $x=u-e$ where $u\in U(R)$ and $e\in Idem(R)$. Let $f=u^{-1}(1-e)u$ then $f^2=f$. Now\begin{align*}
                                                                                              u(x+f) &=u(u-e+u^{-1}(1-e)u) \\
                                                                                                     &=(u-e)^2+(u-e)\\
                                                                                                     &=x^2+x
                                                                                            \end{align*}
$\therefore \,\,\, x+f\in R(x^2+x)$. So $x$ satisfies weakly exchange property by Lemma 2.1 \cite{5}. Hence $R$ is weakly exchange ring.
 $(\Leftarrow)$ Suppose $x\in R$ then there exists an idempotent $e\in Rx$ such that  $1-e \in R(1-x)$ or $1-e \in R(1+x)$.\\
 Case $1$: Suppose $1-e \in R(1+x)$ and $e=ax$, $1-e=b(1+x)$, for some $a,b\in R$. Assume that $ea=a$ and $(1-e)b=b$ so that $axa=ea=a$ and $b(1+x)b=b$. Here $ax,\,xa,\,b(1+x),\,(1+x)b$ all are central idempotent and $ax=(ax)(ax)=(ax)(xa)=x(ax)a=xa$ similarly $(1+x)b=b(1+x)$. Now $(a+b)(x+(1-e))=ax+bx+a(1-e)+b(1-e)=1$ so $x+(1-e)$ is unit. Hence $x$ is weakly clean element.\\
 Case $2$: If $1-e\in R(1-x)$ then clearly $x$ is weakly clean by Proposition 1.8 \cite{8}.
 \end{proof}
 A polynomial ring over a commutative ring $R$ can never be weakly $r$-clean as $x$ can not be represented in weakly clean decomposition. N. Ashrafi and E. Nasibi \cite{2}, showed that $R$ is $r$-clean \textit{if and only if} $R[[x;\alpha]]$ is $r$-clean, where $R[[x;\alpha]]$ is the ring of skew formal power series over $R$ \textit{i.e.} all formal power series in $x$ with coefficients from $R$ and $\alpha$ is an ring endomorphism. Multiplication is defined by $xr=\alpha (r)x$, for all $r\in R$. In particular $R[[x]]=R[[x;I_R]]$ is the ring of formal power series over $R$.
 \begin{Prop}
   Let $R$ be an abelian ring and $\alpha$ be an endomorphism of $R$. Then the following are equivalent.
   \begin{enumerate}
     \item $R$ is weakly $r$-clean ring.
     \item The formal power series ring $R[[x]]$ over $R$ is weakly $r$-clean.
     \item The skew power series ring $R[[x;\alpha]]$ over $R$ is weakly $r$-clean.
   \end{enumerate}
 \end{Prop}
 \begin{proof}
              Follows from Theorem 2.8, and similar results of weakly clean rings.
               \end{proof}
Consider a ring $R[D,C]$ constructed in \cite{10}, for any subring $C$ of a ring $D$. The following result gives the relation between $R[D,C]$ with $D$ and $C$, about weakly $r$-clean ness.
\begin{Prop}
  Let $C$ be a subring of a ring $D$ then $R[D,C]$ is weakly $r$-clean ring \textit{if and only if} $D$ is $r$-clean and $C$ is weakly $r$-clean ring.
\end{Prop}
\begin{proof}
  $(\Rightarrow)$ As $R[D,C]=D\otimes D$ so by Theorem $2.3$ $D$ is $r$-clean. Clearly $C$ is weakly $r$-clean ring by Theorem $2.2$.\\
$(\Leftarrow)$ Let $D$ be $r$-clean and $C$ is weakly $r$-clean ring. Let $x=(a_1,a_2,\cdot\cdot\cdot, a_n, c, c, c, \cdot\cdot\cdot)\in R[D,C]$. Since $C$ is weakly $r$-clean ring, so $c=r\pm e$, where $r\in Reg(R)$ and $e\in Idem(R)$. If $c=r+e$, write $a_i=r_i+e_i$ then $x=\overline{r}+\overline{e}$, where $\overline{r}=(r_1, r_2, \cdot\cdot\cdot , r_n, r, r, \cdot\cdot\cdot)\in Reg(R[D,C])$ and $\overline{e}=(e_1, e_2, \cdot\cdot\cdot , e_n, e, e, \cdot\cdot\cdot)\in Idem(R[D,C])$. If $c=r-e$, write $a_i=r_i-e_i$ then $x=\overline{r}-\overline{e}$, where $\overline{r}=(r_1, r_2, \cdot\cdot\cdot , r_n, r, r, \cdot\cdot\cdot)\in Reg(R[D,C])$ and $\overline{e}=(e_1, e_2, \cdot\cdot\cdot , e_n, e, e, \cdot\cdot\cdot)\in Idem(R[D,C])$.
\end{proof}
Clearly every abelian semi regular  ring is weakly clean and hence  weakly $r$-clean. However the converse is not true as shown by the following example.
\begin{example}
  Let $\mathbb{Q}$ be a field of rational numbers and $L$ the ring of all rational numbers with odd denominators. 
  Then by example 2.7 \cite{2} and Theorem $2.8$, $R[\mathbb{Q},L]$ is commutative exchange ring and hence $R[\mathbb{Q},L]$ is weakly $r$-clean ring but not semi regular.
\end{example}
If $R$ is weakly $r$-clean ring then $R/I$ is weakly $r$-clean being homomorphic image of weakly $r$-clean ring. The theorem is a partial converse of this statement.
\begin{Thm}
  Let $I$ be a regular ideal of a ring $R$, where idempotents can be lifted modulo $I$. Then $R$ is weakly $r$-clean \textit{if and only if} $R/I$ is weakly $r$-clean.
\end{Thm}

\begin{proof}
  Following  theorem 2.8 \cite{2}, it is enough to show that for any $a\in R$ if $a+I$ has type $2$ weakly $r$-clean decomposition in $R/I$ then $a$ is weakly $r$-clean in $R$. Let $a+I \in R/I$ has type $2$ weakly $r$-clean decompositions in $R/I$ then there exists idempotent $e+I \in R/I$ such that $(a+I)+(e+I)\in Reg(R/I)$ \textit{i.e.} $(a+e)+I\in Reg(R/I)$. Therefore $((a+e)+I)(x+I)((a+e)+I)=(a+e)+I$, for some $x\in R$, so $(a+e)x(a+e)-(a+e)\in I$ but $I$ is regular ideal, Hence $(a+e)$ is regular by Lemma 1, \cite{4} . Since idempotents can be lifted modulo $I$, we may assume that $e$ is an idempotent of $R$.
\end{proof}
\begin{Prop}
  Let $M$ be a $A-B$ bi-module. If $T=\left(
                                      \begin{array}{cc}
                                        A & 0 \\
                                        M & B \\
                                      \end{array}
                                    \right)
  $, a formal triangular matrix ring is weakly $r$-clean then $A$ and $B$ are weakly $r$-clean ring.
\end{Prop}
\begin{proof}
  Let $a\in A$ and $b\in B$ and $m\in M$, consider $t=\left(
                                                            \begin{array}{cc}
                                                              a & 0 \\
                                                              m & b \\
                                                            \end{array}
                                                          \right) \in T
  $. \\
  Case $I$: If $t$ is $r$-clean then clearly $a$ and $b$ both are $r$-clean in their respective rings $A$ and $B$ by Theorem 16, \cite{3} .\\
  Case $II$: If $t=-\left(
                     \begin{array}{cc}
                       f_1 & 0 \\
                       f_2 & f_3 \\
                     \end{array}
                   \right)+\left(
                     \begin{array}{cc}
                       r_1 & 0 \\
                       r_2 & r_3 \\
                     \end{array}
                   \right)
  $ where $\left(
                     \begin{array}{cc}
                       f_1 & 0 \\
                       f_2 & f_3 \\
                     \end{array}
                   \right)^2=\left(
                     \begin{array}{cc}
                       f_1 & 0 \\
                       f_2 & f_3 \\
                     \end{array}
                   \right)$ and $\left(
                     \begin{array}{cc}
                       r_1 & 0 \\
                       r_2 & r_3 \\
                     \end{array}
                   \right)\in Reg(T)$. It is simple calculation to show that $f_1$ and $f_3$ are idempotents in $A$ and $B$ respectively and $r_1$ and $r_3$ are regular elements in $A$ and $B$ respectively. Finally $a=-f_1+r_1$ and $b=-f_3+r_3$. Hence $A$ and $B$ both are weakly $r$-clean rings.
\end{proof}
In \cite{10}, T. Kosan, S. Sahinkaya and Y. Zhou showed that  the center of weakly clean rings need not be weakly clean. In that example $End(_{R}E(R))$ is weakly clean whereas the center is not weakly clean as $R$ is not weakly clean. But since $R$ is abelian in that example so by Theorem 2.8, the center of $End(_{R}E(R))$ is not weakly $r$-clean. Hence the center of weakly $r$-clean ring need not be weakly $r$-clean. 
\begin{Thm}
  Let $R$ be a weakly $r$-clean ring with no nontrivial idempotents. Then the center of $R$ is also a weakly $r$-clean ring.
\end{Thm}
\begin{proof}
Let $R$ be weakly $r$-clean ring and $Z(R)$ be the center of $R$. Let $x\in Z(R)$ then there exists regular element $r\in R$ such that either $x=r$ or $x=r-1$ or $x=r+1$. If $x=r$ then clearly $x$ is weakly $r$-clean in $Z(R)$. If $x=r\pm 1$ then $r\in Z(R)$ as $x\pm 1 \in Z(R)$. Hence $Z(R)$ is weakly $r$-clean ring.
\end{proof}
  A ring $R$ is called graded ring(or more precisely, $\mathbb{Z}$-graded) if there exists a family of subgroups $\{R_n\}_{n\in Z}$ of $R$ such that $R=\underset{n}{\oplus}R_n$ (as abelian groups) and $R_nR_m\subseteq R_{n+m}$, for all $n,m\in \mathbb{Z}$. The relation of graded ring and weakly $r$-clean ring is given below.
  \begin{Prop}
    Let $R=R_0\oplus R_1\oplus R_2\oplus \cdot \cdot \cdot $ be a graded ring then the following hold:
    \begin{enumerate}
      \item If $R$ is weakly $r$-clean ring then $R_0$ is weakly $r$-clean ring.
      \item If $R_0$ is weakly $r$-clean ring and each $R_n$ is a torsion-less $R_0$ module, \textit{i.e.} $r\notin Z(R_0)\Rightarrow r\notin Z(R_n)$, then $R$ is weakly $r$-clean ring.
    \end{enumerate}
  \end{Prop}
  \begin{proof}
    \begin{enumerate}
      \item Let $r_0\in R_0$, so $r_0=r+e$ or $r_0=r-e$, where $r\in Reg(R)$ and $e\in Idem(R)$. Since $Idem(R)=Idem(R_0)$, so $e\in R_0$ and $r\in R_0\cap Reg(R)\subseteq Reg(R_0)$. Hence $R_0$ is weakly $r$-clean ring.
      \item Let $x=x_0+x_1+x_2+\cdot \cdot \cdot +x_n \in R$, where $x_i\in R_i$. Write $x_0=r_0\pm e_0$, where $r_0\in Reg(R_0)$ and $e_0\in Idem(R_0)=Idem(R)$. Put $x'=r_0+x_1+x_2+\cdot \cdot \cdot +x_n$, so $x=x'\pm e_0$ where $e_0\in Idem(R_)$. If $x'\in Z(R)$ then there exists a non zero homogeneous $t\in R_n$ with $tx'=0$. But then $r_0t=0$, a contradiction.
    \end{enumerate}
  \end{proof}
\section{Weakly $g(x)$-$r$-clean rings}
\begin{Def}
  Let $g(x)$ be a fixed polynomial in $C(R)[x]$. An element $x\in R$ is called weakly $g(x)$-$r$-clean if $x=r\pm e$, where $r\in Reg(R)$ and $g(e)=0$. We say that $R$ is weakly $g(x)$-$r$-clean if every element is weakly $g(x)$-$r$-clean.
\end{Def}
If $g(x)=x^2-x$, then weakly $g(x)$-$r$-clean ring is similar to weakly $r$-clean ring. An element $s$ of $R$ is called root of the polynomial $g(x)\in C(R)[x]$ if $g(s)=0$. For $g(x)=x^2-x$, the ring $\mathbb{Z}_{(3)}\cap \mathbb{Z}_{(5)}$ is weakly $g(x)$-$r$-clean ring but not $g(x)$-$r$-clean ring. \par
Let $R$ and $S$ be rings and $\theta :C(R) \rightarrow C(S)$ be a ring homomorphism with $\theta (1)=1$. Then $\theta$ induces a map $\theta '$ from $C(R)[x]$ to $C(S)[x]$ such that for $g(x)=\sum_{i=0}^{n} a_ix^i \in C(R)[x]$, $\theta '(g(x)):=\sum_{i=0}^{n} \theta (a_i)x^i \in C(S)[x]$.
\begin{Prop}
  Let $\theta :R\rightarrow S$ be a ring epimorphism. If $R$ is weakly $g(x)$-$r$-clean ring then $S$ is weakly $\theta '(g(x))$-$r$-clean ring.
\end{Prop}
\begin{proof}
  Let $g(x)=a_0+a_1x+\cdot\cdot\cdot +a_n x^n \in C(R)[x]$, then $\theta '(g(x))=\theta(a_0)+\theta( a_1)x+\cdot\cdot\cdot +\theta(a_n) x^n \in C(S)[x]$. Since $\theta$ is ring epimorphism so for any $s\in S$, there exists $x\in R$ such that $\theta(x)=s$. Let $x=r\pm s_0$ where $r\in Reg(R)$ and $g(s_0)=0$, as $R$ is weakly $g(x)$-$r$-clean ring. Now $s=\theta(x)=\theta(r\pm s_0)=\theta(r)\pm \theta(s_0)$. Clearly $\theta(r)\in Reg(S)$ and \begin{align*}
        \theta '(g(\theta(s_0))) &=\theta(a_0)+\theta(a_1)\theta(s_0)+\cdot\cdot\cdot +\theta(a_n)(\theta(s_0))^n \\
                                 &=\theta(a_0+a_1s_0+\cdot\cdot\cdot +a_ns_0^n) \\
                                 &=\theta(0)\\
                                 &=0
      \end{align*}
  Hence $S$ is $\theta'(g(x))$-$r$-clean ring.
\end{proof}
Next we extend the Theorem 2.3 to $g(x)$-$r$-clean ring.
\begin{Thm}
  Let $g(x)\in \mathbb{Z}[x]$ and $\{R_i\}_{i\in I}$ be a family of rings. Then $\underset {i\in I}{\prod}R_i$ is $g(x)$-$r$-clean ring \textit{if and only if} $R_i$'s are weakly $g(x)$-$r$-clean ring and at most one $R_{\alpha}$ is not $g(x)$-$r$-clean ring.
\end{Thm}
\begin{proof}$(\Rightarrow)$ Similar to Theorem 2.3.\\
  $(\Leftarrow)$ If each $R_\alpha$ is $g(x)$-$r$-clean ring, then $R=\prod R_\alpha$ is $g(x)$-$r$-clean. Assume $R_{\alpha_0}$ is weakly $g(x)$-$r$-clean but not $g(x)$-$r$-clean and that all other $R_{\alpha}$'s are $g(x)$-$r$-clean. Let $x=x_\alpha \in R$ so in $R_{\alpha_0}$ we can write $x_{\alpha_0}=r_{\alpha_0}+s_{\alpha_0}$ or $x_{\alpha_0}=r_{\alpha_0}-s_{\alpha_0}$, where $r_{\alpha_0}\in Reg(R_{\alpha_0})$ and $g(s_{\alpha_0})=0$ in $R_{\alpha_0}$ . If $x_{\alpha_0}=r_{\alpha_0}+s_{\alpha_0}$ then for $\alpha \neq \alpha_0$ assume, $x_{\alpha}=r_{\alpha}+s_{\alpha}$ and if $x_{\alpha_0}=r_{\alpha_0}-s_{\alpha_0}$ then for $\alpha \neq \alpha_0$ assume, $x_{\alpha}=r_{\alpha}-s_{\alpha}$, where $r_{\alpha}\in R_{\alpha}$ and $g(s_{\alpha})=0$ in $R_{\alpha}$. Then $r=(r_\alpha)\in Reg(R)$ and $g(s=\{s_{\alpha}\})=a_0\{1_{R_i}\}+a_1\{s_i\}+\cdot\cdot\cdot +a_n\{s_i^n\}=\{a_0+a_1s_i+\cdot\cdot\cdot +a_ns_i^n\}$=0.
\end{proof}
 \begin{Thm}
   Let $R$ be a ring and $a,b\in R$ then $R$ is weakly $(ax^{2n}-bx)$-$r$-clean ring \textit{if and only if} $R$ is weakly $(ax^{2n}+bx)$-$r$-clean ring, for $n\in \mathbb{N}$.
 \end{Thm}
 \begin{proof}
   Suppose $R$ is weakly $(ax^{2n}-bx)$-$r$-clean ring. Since $(as^{2n}-bs)=0\Rightarrow (a(-s)^{2n}+b(-s))=0$. So for $a\in R$, $-a=r\pm s$, where $(as^{2n}-bs)=0$ and $r\in R$. Hence $a=(-r)\pm (-s)$ and $(a(-s)^{2n}+b(-s))=0$. Similarly the converse is also true.
 \end{proof}
 \begin{Prop}
   Let $2\leq n\in \mathbb{N}$. If for every $a\in R$, $a=r\pm t$, where $r\in Reg(R)$ and $t^{n-1}=1$ then $R$ is weakly $(x^n-x)$-$r$-clean ring.
 \end{Prop}
 \begin{proof}
   The proof follows from Theorem 3.4.
 \end{proof}
\begin{Prop}
  Let $g(x)=(x^n-x)\in C(R)[x]$ and $g(e)=0$. Let $a\in e^{n-1}Re^{n-1}$ be strongly $g(x)$-clean in $e^{n-1}Re^{n-1}$ then $a$ is strongly $g(x)$-clean in $R$.
\end{Prop}
\begin{proof}
  Let $a=u+f$, where $u\in U(e^{n-1}Re^{n-1})$ and $g(f)=0$ in $e^{n-1}Re^{n-1}$. Then $uw=e^{n-1}=wu$, for some $w\in U(e^{n-1}Re^{n-1})$. Consider $v=u-(1-e^{n-1})$, then $\{u-(1-e^{n-1})\}\{w-(1-e^{n-1})\}=1$, so $v=u-(1-e^{n-1})$ is a unit in $R$. Now $a-v=f+(1-e^{n-1})$ and $\{f+(1-e^{n-1})\}^n=f^n+(1-e^{n-1})^n=f+(1-e^{n-1})$, so $a$ is $g(x)$-clean ring in $R$.
\end{proof}
\begin{Prop}
  Let $R$ be an abelian ring and $g(x)\in C(R)[x]$ then if $a$ is $g(x)$-clean element in $R$ then $ae^{n-1}$ is also $g(x)$-clean element in $R$ for any root $e$ of g(x).
\end{Prop}
\begin{proof}
  Let $a=u+f$, where $u\in U(R)$ and $g(f)=0$ in $R$. Suppose $e$ be any root of $g(x)$ in $R$. Then $ae^{n-1}=ue^{n-1}+fe^{n-1}$ but $ue^{n-1}\in U(e^{n-1}Re^{n-1})$ and $(fe^{n-1})^n=f^ne^{n-1}=fe^{n-1}$. So $ae^{n-1}$ is $g(x)$-clean in $R$.
\end{proof}

\section{Weakly $\star$-clean rings and $\star$-$r$-clean ring}
A ring $R$ is a $\star$-ring (or ring with involution) if there exists an operation $\star \,\,:\,\,R\rightarrow R$ such that for all $x,y\in R$, $(x+y)^{\star}=x^{\star}+y^{\star}$, $(xy)^{\star}=y^{\star}x^{\star}$ and $(x^{\star})^{\star}=x$. An element $p$ of a $\star$-ring is a projection if $p^2=p=p^{\star}$. Obviously, $0$ and $1$ are projections of any $\star$-ring. Henceforth $P(R)$ will denote the set of all projections in a $\star$-ring.  \par
Here we define the concept of weakly $\star$-clean ring and discuss some properties of weakly $\star$-clean ring.
\begin{Def}
  An element $x$ of a $\star$-ring $R$ is said to be weakly $\star$-clean if $x=u + p$ or $x = u - p$, where $u\in U(R)$ and $p\in P(R)$. A $\star$-ring $R$ is said to be weakly $\star$-clean ring if every elements are weakly $\star$-clean.
\end{Def}
\begin{example}
  \begin{enumerate}
    \item  Units, elements in $J(R)$ and nilpotents of a $\star$-ring $R$ are weakly $\star$-clean.
    \item  Idempotents of a $\star$-regular rings are weakly $\star$-clean.
  \end{enumerate}
\end{example}
\begin{Lem}
  Let $R$ be a boolean $\star$-ring. Then $R$ is weakly $\star$-clean \textit{if and only if} $\star =1_R$ is the identity map of $R$.
\end{Lem}
\begin{proof}
  It is clear that boolean rings are clean. Suppose that $R$ is weakly $\star$-clean. Given any $a\in R$, we have $-a=u\pm p=\pm p+1=\pm p-1$, for some $p\in P(R)$. So we have $a=1\pm p \in P(R)$. Hence $a^{\star}=a$, which implies $\star=1_R$. Conversely if $\star=1_R$ then every idempotent of $R$ is a projection. Thus, $R$ is a weakly $\star$-clean.
\end{proof}
\begin{example}
  $R=\mathbb{Z}_2\oplus \mathbb{Z}_2$ with involution $\star$ defined by $(a,b)^{\star}=(b,a)$ is weakly clean but not weakly $\star$-clean.
\end{example}
\begin{Lem}
  Let $R$ be a $\star$-ring. If $2\in U(R)$, then for any $u^2=1,\,u^{\star}=u\in R$ \textit{if and only if} every idempotent of $R$ is a projection.
\end{Lem}
\begin{proof}

Proof is given in Lemma 2.3 \cite{11}.
\end{proof}
\begin{Cor}
  Let $R$ be a $\star$-ring with $2\in U(R)$. The following are equivalent:
\begin{enumerate}
  \item $R$ is weakly clean and every unit of $R$ is self-adjoint.(\textit{i.e.}, $u^{\star}=u$ for every unit $u$ ).
  \item $R$ is weakly $\star$-clean and $\star=1_R$.
\end{enumerate}
\end{Cor}
\begin{proof}
$(1)\Rightarrow (2)$. Let $a\in R$. Then $a=u\pm p$, for some $p\in P(R)$ and $u\in U(R)$. Note that $(1-2p)\in U(R)$, so $(1-2p)^{\star}=(1-2p)\Rightarrow p^{\star}=p$ by Lemma 4.5. Also $a^{\star}=(u\pm p)^{\star}=u^{\star}\pm p^{\star}=u\pm p=a$.\\
$(2)\Rightarrow (1).$ is trivial.
\end{proof}
For a $\star$-ring $R$, an element $x\in R$ is called \textit{self adjoint square root of 1} if $x^2=1$ and $x^{\star}=x$.
\begin{Thm}
  Let $R$ be a $\star$-ring, the following are equivalent:\\
(1) $R$ is weakly $\star$-clean and $2\in U(R)$.\\
(2) Every element of $R$ is a sum of unit and a self adjoint square root of 1 or an element of the form $2p+1$, where $p^2=p=p^{\star}$.
\end{Thm}
\begin{proof}
  $(1)\Rightarrow (2)$. Consider $a\in R$, then $\frac{a-1}{2}=u\pm p$, where $u\in U(R)$ and $p\in P(R)$. If $\frac{a-1}{2}=u-p$, then $a=2u+(1-2p)$, where $2u\in U(R)$ and $1-2p$ is a self adjoint square root of 1. If $\frac{a-1}{2}=u+p$, then $a=2u+(1+2p)$, where $p^2=p=p^{\star}$. \\
$(2)\Rightarrow (1)$ First we show that $2\in (R)$. By assumption, $1=u+x$ or $1=u+(2p+1)$, where $u\in U(R)$, $x$ is self adjoint square root of $1$ and $p\in P(R)$. If $1=u+x$, then clearly $2\in U(R)$ by Theorem 2.5 \cite{11}. If $1=u+(2p+1)$, then $u=-2p$, so $2\in U(R)$. Next for showing $R$ is weakly $\star$-clean ring, let $a\in R$, so $2a+1=u+x$ or $2a+1=u+(2p+1)$, where $u\in U(R)$, $x$ is self adjoint square root of $1$ and $p\in P(R)$.\\
Case I: If $2a+1=u+x$, then $a=\frac{u}{2}-\frac{1-x}{2}$. Since $(\frac{1-x}{2})^2=\frac{1-x}{2}=(\frac{1-x}{2})^{\star}$ and $2\in (R)$, so $a$ is weakly $\star$-clean element. \\
Case II: If $2a+1=u+(2p+1)$, then $a=\frac{u}{2}+p$, a $\star$-clean element.
\end{proof}
\begin{Lem}
  Let $R$ be weakly $\star$-clean ring. If $I$ is a $\star$ invariant ideal of $R$, then $R/I$ is weakly $\star$-clean. In particular, $R/J(R)$ is weakly $\star$-clean ring.
\end{Lem}
\begin{proof}
  The result follows from 
  Lemma 2.7 $\cite{11}$.
\end{proof}
Let $R$ be a $\star$-ring. Then $\star$ induces an involution in the power series ring $R[[x]]$, defined by $(\sum_{i=0}^{\infty}a_ix^i)^{\star}=\sum_{i=0}^{\infty}a_i^{\star}x^i$.
\begin{Prop}
  Let $R$ be a $\star$-ring. Then $R[[x]]$ is weakly $\star$-clean \textit{if and only if} $R$ is weakly $\star$-clean.
\end{Prop}
\begin{proof}
  Let $R[[x]]$ be weakly $\star$-clean. Since $R\cong R[[x]]/<x>$ and $<x>$ is $\star$-invariant ideal of $R[[x]]$. So $R$ is weakly $\star$-clean ring. Conversely, suppose that $R$ is weakly $\star$-clean ring. Let $g(x)=\sum_{i=0}^{\infty}a_ix^i \in R[[x]]$. Write $a_0=u\pm p$ with $p\in P(R)$ and $u\in U(R)$. Then $g(x)=\pm p+(u+\sum_{i=1}^{\infty}a_ix^i)$, where $p\in P(R)\subseteq P(R[[x]])$ and $u+\sum_{i=1}^{\infty}a_ix^i\in U(R[[x]])$. Hence $g(x)$ is weakly $\star$-clean in $R[[x]]$.
\end{proof}
\begin{Thm}
  Homomorphic image of weakly $\star$-clean ring is weakly $\star$-clean.
\end{Thm}
Similarly we extend the Theorem 2.3 to $\star$-clean ring given below.
\begin{Thm}
  Let $\{R_\alpha\}$ be a collection of $\star$-rings. Then the direct product $R=\underset{\alpha}{\prod}R_{\alpha}$ is weakly $\star$-clean \textit{if and only if} each $R_{\alpha}$ is weakly $\star$-clean ring and at most one $R_{\alpha}$ is not $\star$-clean.
\end{Thm}
\begin{proof} Similar to the proof of Theorem 2.3.
\end{proof}
\begin{Def}
  An element $x$ in a $\star$-ring $R$ is said to be $\star$-$r$-clean if $x=r+p$ where $r\in Reg(R)$ and $p\in P(R)$. A $\star$-ring $R$ is said to be $\star$-$r$-clean ring if every element of $R$ is $\star$-$r$-clean.
\end{Def}
\begin{Prop}
  Let $R$ be a $\star$-ring and $e\in P(R)$. If $a\in eRe$ is strongly $\star$-clean in $eRe$ then $a$ is strongly $\star$-clean in $R$.
\end{Prop}
\begin{proof}
Let $a=f+v$, where $f\in P(eRe)$ and $v\in U(eRe)$, so there exists $w\in U(eRe)$ such that $vw=e=wv$. Clearly $(v-(1-e))(w-(1-e))=1$ implies $v-(1-e)\in U(R)$ and $a-(v-(1-e))=f+(1-e)\in P(R)$. Hence $a$ is strongly $\star$ clean in $R$.
\end{proof}
\begin{Prop}
   Let $R$ be an abelian $\star$-ring. If $a$ is $\star$-clean element in $R$ then $ae$ is $\star$-clean for any $e\in P(R)$.
\end{Prop}
\begin{proof}
  Let $a=u+p_1$, where $u\in U(R)$ and $p_1\in P(R)$. Now $ae=ue+p_1e$. Clearly $ue\in U(eRe)$ and $p_1e\in P(eRe)$ imply $ae$ is strongly $\star$-clean in $eRe$, so $ae$ is strongly $\star$-clean in $R$.
\end{proof}
\begin{Prop}
  Let $R$ be an abelian $\star$-ring and $a$ be a $\star$-clean element in $R$ and $e\in P(R)$. If $-a$ is $\star$-clean then $a+e$ is also $\star$-clean.
\end{Prop}
\begin{proof}
  Clearly $a$ and $a+1$ are $\star$-clean, as $a$ is $\star$-clean implies $1-a$ is so. Let $a=u+f$ and $1+a=v+g$, where $f,g\in P(R)$ and $u,v \in U(R)$. $a+e=(1+a)e+a(1-e)=(ve+u(1-e))+(ge+f(1-e))$. But $(ve+u(1-e))\in U(R)$ and $(ge+f(1-e))\in P(R)$. So $a+e$ is $\star$-clean.
\end{proof}
\begin{Lem}
  Let $R$ be an abelian $\star$-ring where every idempotent of the form $ry$ or $yr$ is projection, for any regular element $r$, then $r$ is $\star$-clean.
\end{Lem}
\begin{proof}
  Let $r\in Reg(R)$ then $r=ryr$, for some $y\in R$. So $f=yr\in P(R)$. Let $e=f+(1-f)rf=(y+(1-f)ry)r=ar $, where $a=y+(1-f)ry$. Here $e^2=e\in P(R)$ and $(1-e)=(1-f)(1-r)=b(1-r)\in R(1-r)$, where $b=(1-f)$. Assume $ea=a$ so that $ara=a$. Since idempotents are central so $ra=r(ar)a=ra(ra)=(ra)ra=a(ra)r=ar$. Similarly $(1-r)b=b(1-r)$, where it is assumed that $(1-e)b=b$. By Proposition 1.8 \cite{8}, $a-b$ is the inverse of $r-(1-e)$ and so $r$ is $\star$-clean.
\end{proof}

\begin{Thm}
  Let $R$ be an abelian $\star$-clean, where any idempotent of the form $e=ry$ or $yr$ is projection, for any $r\in Reg(R)$. Then $R$ is $\star$-$r$-clean \textit{if and only if} $R$ is $\star$-clean.
\end{Thm}
\begin{proof}
  $(\Leftarrow)$ Obviously $R$ is $\star$-clean $\Rightarrow$ $R$ is $\star$-$r$-clean. \par
  $(\Rightarrow)$ Let $R$ be a $\star$-$r$-clean ring and $x\in R$, then $x=r+e'$, where $e'\in P(R)$ and $r\in Reg(R)$. Therefore $r=ryr$, for some $y\in R$. Taking $e=ry$ we see that $(re+(1-e))(ye+(1-e))=1$ and $e\in P(R)$. Hence $u=re+(1-e)$ is a unit and $r=eu$. Set $f=1-e$ then $-(eu+f)$ is a unit and $f\in P(R)$. So, $-r=-(ue+f)+f$ is $\star$-clean. Also by Lemma 4.16, $r$ is $\star$-clean. Hence by Proposition 4.15, $x=r+e'$ is $\star$-clean.
\end{proof}

\end{document}